\documentclass[12pt,a4paper]{article}

\RequirePackage{etex}
\usepackage[utf8]{inputenc}
\usepackage{amsmath, amssymb, amsthm, amsfonts, mathrsfs}
\usepackage{bbm, commath}
\usepackage{multicol}
\usepackage{blindtext}
\usepackage{graphicx}
\usepackage{tikz-network}
\usetikzlibrary {positioning, patterns, calc}
\usetikzlibrary{graphs, graphs.standard, quotes}
\usepackage{authblk}
\usepackage{enumitem}
\usepackage[colorlinks=true,linkcolor=black,citecolor=black]{hyperref}
\usepackage{tabularx}
\usepackage{array}

\usepackage{natbib}
 \usepackage{setspace}
 \let\oldbibitem\bibitem
 \renewcommand{\bibitem}{\setstretch{1.2}\oldbibitem}

\newtheorem{thm}{Theorem}

\newtheorem{prop}{Proposition}
\newtheorem{lem}{Lemma}
\newtheorem{cor}{Corollary}
\newtheorem{exm}{Example}

\newtheorem{rem}{Remark}
\def \e {{\bf e}}
\def \f {{\bf f}}
\def \g {{\bf g}}
\def \h {{\bf h}}
\def \u {{\bf u}}
\def \v {{\bf v}}
\def \w {{\bf w}}

\def \e {{\mathbf e}}

\newcommand{\ob}[1]{\left(#1\right)}
\newcommand{\cb}[1]{\left\lbrace #1\right\rbrace}
\newcommand{\tb}[1]{\left[#1\right]}

\title{\it Fractional revival in complementary prisms of graphs}
\author[1]{Sarojini Mohapatra}
\author[2]{Hiranmoy Pal}
\affil[1,2]{National Institute of Technology Rourkela, India-769008 palh@nitrkl.ac.in}
\date{\today}

\begin{document}
	
	\maketitle

	%%%%%%%%%%% Abstract %%%%%%%%%%%
	
		\begin{abstract}
        The complementary prism $G\overline{G}$ of a graph $G$ is obtained from the disjoint union of $G$ and its complement $\overline{G}$ by adding an edge between each vertex $a$ in $G$ and its copy $a'$ in $\overline{G}.$ This paper explores a general framework for studying fractional revival with respect to real symmetric matrices with a block structure. The framework is then used to show that, for a fixed state $\u$ in $G$ orthogonal to the all-one vector, the complementary prism $G\overline{G}$ exhibits fractional revival from the state $[\u,\mathbf{0}]^T$ with respect to the adjacency, Laplacian, and signless Laplacian matrices. We further characterize perfect pair state transfer in the complementary prism of a complete graph and establish the existence of perfect pair and plus state transfer in the complementary prism of a complete bipartite graph. 

\vspace{0.5cm}      
    \noindent{\it Keywords:} Continuous-time quantum walk,  Spectra of graphs, Fractional revival, Perfect state transfer, Complementary prism of a graph. \\\\
    {\it MSC: 15A16, 05C50, 81P45.}
	\end{abstract}

	%%%%%%%%%%% Introduction %%%%%%%%%%%
	\newpage

\section{Introduction} 
The reliable transfer of quantum information and the generation of entanglement are important problems in quantum information theory. 
Continuous-time quantum walks provide a natural framework for studying quantum state transfer in quantum spin networks. Such networks can be modeled by an undirected weighted graph where vertices represent qubits and the interactions between them are represented by edges, with the edge weight corresponding to the coupling strength between the qubits. Let $M(G)$ be a real symmetric matrix associated with a graph $G.$ A \emph{continuous-time quantum walk} on $G$ with $M(G)$ as its Hamiltonian is governed by the transition matrix 
 \[U_{M(G)}(t)=\exp{(itM(G))},\quad \text{where $t\in\mathbb{R}.$} \] 
 When the context is clear, we write $M(G)$ as $M.$
A \emph{real pure state} is represented by a unit vector in $\mathbb{R}^n$ \cite{god7}. For each vertex $a$ of $G,$ the characteristic vector $\e_a$ is called the \emph{vertex state} associated with $a.$ Real pure states of the form $\frac{1}{\sqrt{2}}\ob{\e_a-\e_b}$ and $\frac{1}{\sqrt{2}}\ob{\e_a+\e_b}$ are called \emph{pair state} and \emph{plus state}, respectively. A graph $G$ is said to exhibit \emph{perfect state transfer} (PST) between two linearly independent real pure states $\u$ and $\v$ if there exists a time $\tau>0$ and a complex number $\gamma$ of unit modulus such that
\begin{equation}\label{6e1}
U_{M}(\tau)\u=\gamma\v.
\end{equation}
  PST between vertex states, pair states, and plus states is called \emph{vertex PST}, \emph{pair PST}, and \emph{plus PST}, respectively. If $\u=\v$ in \eqref{6e1}, then the state $\u$ is said to be \emph{periodic} in $G.$ 
PST in quantum spin networks has been extensively studied since its introduction by Bose \cite{bose}, with much of the literature devoted to the existence and characterization of vertex PST \cite{alv, bas1, pal9, cou7, god1, kirk3}.
However, Godsil \cite[Corollary 6.2]{god2} showed that vertex PST is a rare phenomenon in finite unweighted graphs,  motivating the investigation of PST beyond vertex states. The concept of pair PST  with respect to the Laplacian matrix was first introduced by Chen and Godsil in \cite{che1}, which shows that on a fixed number of vertices, there are more graphs with pair PST than graphs with vertex PST. Further studies have considered PST between more general states, including pair and plus states \cite{jia, ojha1, pal10, wang}, $s$-pair states \cite{kim}, and real pure states \cite{ god8, god7, pal11}.

Fractional revival is another relevant quantum transport phenomenon that generalizes the study of PST and plays an important role in entanglement generation. 
Let $\u$ and $\v$ be two linearly independent states in a graph $G.$ The graph $G$ is said to exhibit \emph{fractional revival} from $\u$ to $\v$ at time $\tau$ if there exist complex scalars $\alpha,\beta$ with $\beta\neq 0$ such that 
\begin{equation}\label{6eq2}
    U_{M}(\tau)\u=\alpha\u+\beta\v.
\end{equation}
 If $\alpha$ becomes zero, then $G$ is said to exhibit PST between $\u$ and $\v$ at $\tau.$  If $\beta=0,$ then $\u$ becomes periodic in $G.$ In particular, if $\u$ and $\v$ in \eqref{6eq2} represent vertex states $\e_a$ and $\e_b,$ respectively, then fractional revival occurs from vertex $a$ to $b$ in $G.$ The existence and characterization of fractional revival from vertex states have been studied for several classes of graphs, including paths and cycles \cite{chan2}, trees \cite{chan3}, double cones \cite{mon4}, threshold graphs \cite{kirk6}, products, covers, and joins of graphs \cite{cha2}. Here, we characterize the existence of fractional revival from a real pure state in the complementary prism of a graph.

  The paper is organized as follows. Section \ref{6s2} contains preliminary results on the existence of PST between real pure states and the spectra of the complementary prism of a graph.  Section \ref{6s3} develops a general framework for the existence of fractional revival from a real pure state relative to a real symmetric matrix with a block structure and provides a graph construction that preserves fractional revival. Section \ref{6s4} applies the framework to characterize fractional revival in the complementary prism of a graph. Further, we investigate pair PST in the complementary prism of a complete graph and pair and plus PST in that of a complete bipartite graph, yielding an infinite family of graphs exhibiting plus PST.
  Throughout the paper, $\mathbf{1}$ denotes the all-ones vector, $I$ the identity matrix, and $\mathbf{0}$ the zero matrix of appropriate order.

\section{Preliminaries}\label{6s2}
The Hamiltonian $M$ corresponding to an undirected graph $G$ is a real symmetric matrix. In this context, we mainly consider $M$ to be the adjacency, Laplacian, or signless Laplacian matrix of $G.$
The graph $G$ is called simple if it has no loops and each of its edges has weight $1.$ For a simple graph $G,$ the adjacency matrix $A$ is defined by $A_{a,b}=1$ if there is an edge between vertices $a$ and $b,$ and $0$ otherwise. The Laplacian and signless Laplacian matrices are defined by $L=D-A$ and $Q=D+A,$ respectively, where $D$ is the degree matrix of $G.$ Since $M$ is real symmetric, it has the spectral decomposition 
   \[M=\sum_{j=1}^{d}\lambda_jE_{\lambda_j},\]
   where $\lambda_1,\lambda_2,\ldots,\lambda_d$ are the distinct eigenvalues of $M$ with corresponding eigenprojection matrices $E_{\lambda_1},E_{\lambda_2},\ldots,E_{\lambda_d}.$ Consequently, the spectral decomposition of the transition matrix $U_M(t)$ is given by
   \[U_M(t)=\sum_{j=1}^{d}\exp{(it\lambda_j)}E_{\lambda_j}.\]
   Therefore, the eigenvalues and eigenvectors of $M$ determine the behavior of the quantum walk. The \emph{eigenvalue support} of a state $\u$ relative to $M,$ denoted by $\sigma_\u(M),$ is the set 
   \[\sigma_\u(M)=\{\lambda_j:E_{\lambda_j}\u\neq 0\}.\]
   \begin{prop}\cite{god7}
    A state $\u\in\mathbb{R}^n\backslash\{0\}$ is a fixed state if and only if $\u$ is an eigenvector for $M$ associated with the lone eigenvalue in $\sigma_\u(M).$
\end{prop}

Two states $\u$ and $\v$ with $\u\neq \pm\v$ are said to be strongly cospectral in $G$ relative to $M$ if 
for each $\lambda_j\in\sigma_{\u}(M),$ we have
$E_{\lambda_j}\u= \pm E_{\lambda_j}\v.$ Let $\sigma_{\u,\v}^+(M)$ and $\sigma_{\u,\v}^-(M)$ denote the sets of eigenvalues $\lambda_j\in\sigma_{\u}(M)$ satisfying $E_{\lambda_j}\u= E_{\lambda_j}\v$ and  $E_{\lambda_j}\u=- E_{\lambda_j}\v,$ respectively. Strong cospectrality is a necessary condition for the existence of PST between $\u$ and $\v$ \cite[Lemma 5.1]{god7}. The following characterization of strong cospectrality, established for pair states with respect to the adjacency matrix in \cite[Proposition 4.1]{saro}, extends to both pair and plus states for any real symmetric matrix $M$ associated with $G.$

\begin{prop}\label{6p2}
     The states $\frac{1}{\sqrt{2}}\ob{\e_a\pm\e_b}$ and $\frac{1}{\sqrt{2}}\ob{\e_c\pm\e_d}$ are strongly cospectral in $G$ if and only if for each eigenvalue $\lambda$ of $M(G),$ there exists $\delta_\lambda\in\{\pm 1\}$ such that every eigenvector $\v$ corresponding to $\lambda$ satisfies 
     \[\v(a)\pm\v(b)=\delta_\lambda[\v(c)\pm\v(d)].\]
 \end{prop}
 If a graph $G$ admits PST between states $\u$ and $\v,$ then both the states are periodic in $G$ \cite[Lemma 5.1]{god7}. The following result gives necessary and sufficient conditions for a state to be periodic in $G.$ 
 \begin{thm}\label{6t3}\cite{god7}
     A state $\u\in\mathbb{R}^n$ is periodic in $G$ if and only if $\sigma_\u(M)$ satisfies the ratio condition. If $|\sigma_\u(M)|\geq 3,$ and $\sigma_\u(M)$ is closed under algebraic conjugates,  then $\u$ is periodic if and only if either  $(i)$ $ \sigma_\u(M)\subseteq \mathbb{Z}$ or $(ii)$ each $\lambda_j\in \sigma_\u(M)$
      is of the form $\lambda_j=\frac{1}{2}(a+b_j\sqrt{\Delta}),$ where $a,b_j,\Delta$ are integers and $\Delta>1$ is square-free.
 \end{thm}

We next include the adjacency, Laplacian, and signless Laplacian spectra of the complementary prism \cite{card1}, which will be used in Section \ref{6s4} to characterize pair and plus PST. 

\subsection{Complementary prism}
The complement of a graph $G,$ denoted by $\overline{G},$ is the graph with the same vertex set as $G,$ in which two distinct vertices are adjacent if and only if they are not adjacent in $G.$ We relabel the vertices of $\overline{G}$ as $V(\overline{G})=\{a':a\in V(G)\},$ where $a'$ denotes the vertex corresponding to $a.$
The complementary prism $G\overline{G}$ of a graph $G$ is obtained from the disjoint union of $G$ and its complement $\overline{G}$ by adding an edge for each pair of vertices $(a,a'),$ where $a$ is in $G$ and its copy $a'$ is in $\overline{G}.$  
For $M\in\{A,L,Q\},$ the matrix representing the complementary prism $G\overline{G}$ of $G$ is defined as follows.
\begin{equation}\label{6eq3}
M(G\overline{G})=
\begin{bmatrix}
M(G)+\delta I & \zeta I\\
\zeta I & M(\overline{G})+\delta I
\end{bmatrix},
\end{equation}
where
\[
\delta=
\begin{cases}
0, & \text {if}~~M=A,\\
1, & \text {if}~~M=L\text{ or }Q,
\end{cases}
\qquad \text{and}\quad
\zeta=
\begin{cases}
-1, & \text{if}~~ M=L,\\
\phantom{-}1, & \text{if}~~ M=A\text{ or }Q.
\end{cases}
\] 
 The relationships between the eigenvalues and eigenvectors of the adjacency, Laplacian, and signless Laplacian matrices of the complementary prism $G\overline{G}$ and those of the graph $G$ are investigated in \cite{card1}. An eigenvalue $\lambda$ of $M(G)$ is said to be a main eigenvalue if there is an associated eigenvector $\v$ which is not orthogonal to $\mathbf{1}.$ Otherwise, $\lambda$ is called non-main. 
 Let $M(G)$ have $d$ distinct eigenvalues. We denote the spectrum of $M(G)$ by 
\[\text{spec}(M(G))=\{\lambda_1^{m_1},\lambda_2^{m_2},\ldots,\lambda_d^{m_d}\},\] 
where $m_j$ indicates the multiplicity of the eigenvalue $\lambda_j.$

\subsection{Spectra of the complementary prism of a graph}
Let $G$ be a connected and regular graph. 
The following result, obtained by combining Theorem 3.2, Theorem 3.3, and Corollary 3.4 of \cite{card1}, determines the full set of eigenvalues and corresponding eigenvectors of the complementary prism $G\overline{G}$ with respect to the adjacency matrix. 
\begin{thm}\cite{card1}\label{5cor1}
    If $G$ is a connected $k$-regular graph on $n$ vertices with spectrum $\operatorname{spec}(A(G))=\cb{k,\lambda_2^{m_2}, \ldots, \lambda_d^{m_d}},$ and $\mathbf{u}_j\perp \mathbf{1}$ is an  eigenvector corresponding to $\lambda_j,$ then the eigenvalues and corresponding eigenvectors of the complementary prism $G\overline{G}$ are as follows. 
    \begin{enumerate}[label=(\roman*)]
        \item $\alpha_{\pm}(\lambda_j)=\dfrac{-1\pm\sqrt{(2\lambda_j+1)^2+4}}{2},$ each with multiplicity $m_j,$ for $2\leq j\leq d$ with eigenvectors $[
            \mathbf{u}_j,
            \ob{\alpha_{\pm}(\lambda_j)-\lambda_j}\mathbf{u}_j]^T$ which are orthogonal to the all one vector.
        
        \item$\beta_{\pm}(n,k)=\dfrac{n-1\pm\sqrt{(n-1-2k)^2+4}}{2}$ with eigenvectors $[
            \mathbf{1},
            \ob{\beta_{\pm}-k}\mathbf{1}]^T.$ Moreover, $\beta_{\pm}(n,k)$ are the two main eigenvalues when $G\overline{G}$ is non-regular.
    \end{enumerate}
\end{thm}
The next two results present the Laplacian and signless Laplacian eigenvalues and corresponding eigenvectors of the complementary prism of a graph and a regular graph, respectively. The expression for the Laplacian eigenvector is obtained from the proof of \cite[Theorem 5.2]{card1}.
\begin{thm}\cite{card1}\label{6th2}
    Let $G$ be a graph on $n$ vertices with Laplacian eigenvalues $\lambda_1\geq \lambda_2\geq\cdots\geq \lambda_{n-1}\geq\lambda_n=0.$ For each $j =1, \ldots, n -1,$ if $\u_j$ is an eigenvector for $\lambda_j,$ then
\[\alpha_{\pm}(\lambda_j) = \dfrac{n + 2 \pm
\sqrt{(n - 2\lambda_j)^2+4}}{2}\]
are eigenvalues of $G\overline{G}$ with associated eigenvectors $\tb{
            \u_j,
            \ob{\lambda_j+1-\alpha_{\pm}\ob{\lambda_j}}\u_j}^T.$ The other eigenvalues of $G\overline{G}$ are $2$ and $0$ with associated eigenvectors  $[
            \mathbf{1},
            -\mathbf{1}]^T$ and $[
            \mathbf{1},
            \mathbf{1}]^T,$ respectively.   
\end{thm}

\begin{thm}\cite{card1}\label{5t5}
    If $G$ is a connected $k$-regular graph on $n$ vertices with $\operatorname{spec}(Q(G))=\cb{2k, \lambda_2^{m_2},\ldots, \lambda_d^{m_d}},$  and $\mathbf{u}_j\perp \mathbf{1}$ is an  eigenvector corresponding to $\lambda_j,$ then the eigenvalues and corresponding eigenvectors of $G\overline{G}$ are as follows.
\begin{enumerate}[label=(\roman*)]
        \item $\alpha_{\pm}(\lambda_j)=\dfrac{n\pm\sqrt{\ob{n-2(\lambda_j+1)}^2+4}}{2},$ each with multiplicity $m_j,$ for $2\leq j\leq d$ with associated eigenvectors $\tb{
            \u_j,
            \ob{\alpha_{\pm}(\lambda_j)-(\lambda_j+1)}\u_j}^T$ which are orthogonal to the all one vector.
        
        \item$\beta_{\pm}(n,k)=n\pm\sqrt{(n-1-2k)^2+1}$ with eigenvectors $\tb{
            \mathbf{1},
            \ob{\beta_{\pm}-2k-1}\mathbf{1}}^T.$ $\beta_{\pm}(n,k)$ are the two main eigenvalues when $G\overline{G}$ is non-regular.

\end{enumerate}
\end{thm}

 \section{Block matrices and fractional revival}\label{6s3}
Unlike the complementary prism, several graph operations, including the join, blow-up, and Cartesian product of graphs, are associated with matrices having a block structure. These constructions motivate the development of a general framework for studying fractional revival in real symmetric matrices with a block structure. This framework builds on the concept of quotient matrices studied in the context of $s$-pair state transfer \cite{kim}.

\begin{lem}\label{6l1}
    Let \[M=\begin{bmatrix}
        M_1 & C_{12} & C_{13} &\cdots & C_{1k}\\
        C_{12}^T & M_2 & C_{23} &\cdots & C_{2k}\\ 
        C_{13}^T & C_{23}^T & M_3  &\cdots & C_{3k}\\
        \vdots & \vdots & \vdots & \ddots & \vdots\\
         C_{1k}^T & C_{2k}^T & C_{3k}^T & \cdots & M_k
        \end{bmatrix}\] be a real symmetric matrix, where $M_j$ is an $n_j\times n_j$ symmetric matrix and $C_{ij}$ is a matrix of appropriate order. Suppose there exist nonzero vectors $\w_j\in\mathbb{R}^{n_j},$ $1\leq j\leq k,$ such that 
        \[M_j\w_j=\lambda_j\w_j, \quad   C_{ij}\w_j=\alpha_{ij}\w_i, \quad \text{and}\quad C_{ij}^T\w_i=\beta_{ij}\w_j  \quad \text{for all}\quad 1\leq i< j\leq k.    \]
        Define $\u_j=[0,\ldots,0,\w_j,0,\ldots,0]^T,$ where $\w_j$ occupies the $jth$ block.
        Then the following holds.
        \begin{enumerate}
            \item $W=span\{\u_1,\u_2,\ldots,\u_k\}$ is an invariant subspace of $M.$ 
            \item If $\mathcal{Q}=[\u_1~~ \u_2~~ \cdots~~ \u_k],$ then $M\mathcal{Q}=\mathcal{Q}H,$ where $H=(h_{ij})_{i,j=1}^k$ and  \[h_{ij}=\left\{ \begin{array}{rcl}
            
 \alpha_{ij}, & \text{if\quad $i<j$,} \\
 \lambda_j, & \text{if \quad$i=j$,}\\

 \beta_{ji}, & \text{if \quad $i>j.$}
 \end{array}\right.\]
        \end{enumerate}
\end{lem}
\begin{proof}
   For each $j,$ the $ith$ block of $M\u_j$ is given by
   \[(M\u_j)_i=\left\{ \begin{array}{rcl}
 C_{ij}\w_j, & \text{if} \quad i<j,\\ 
 M_j\w_j, & \text{if} \quad i=j, \\
 C_{ji}^T\w_j, & \text{if} \quad i>j.
 \end{array}\right.\]
 Hence, $M\u_j=\displaystyle\sum_{i<j}\alpha_{ij}\u_i+\lambda_j\u_j+\sum_{i>j}\beta_{ji}\u_i,$ which proves the first part. If $\h_j$
denotes the $jth$ column of $H,$ then the expression for $M\u_j$ shows that
$M\u_j=\mathcal{Q}\h_j,$
for every $j,$ and hence $M\mathcal{Q}=\mathcal{Q}H.$
 \end{proof}

If all $\w_j$'s in Lemma \ref{6l1} are unit vectors, then $\alpha_{ij}=\beta_{ij},$ and we have the following observation.

\begin{thm}\label{6t1}
    Suppose the premise of Lemma \ref{6l1} holds with each $\u_j$ being a unit vector. Then the graph associated with matrix $H$ exhibits fractional revival from vertex $a$ to $b$ if and only if the graph associated with matrix $M$ exhibits fractional revival from $\u_a$ to $\u_b.$ 
\end{thm}
\begin{proof}
   By Lemma \ref{6l1}(2), we have  $M\mathcal{Q}=\mathcal{Q}H,$ which implies $M^k\mathcal{Q}=\mathcal{Q}H^k$ for all $k\geq 0.$ Therefore, for all $t\in\mathbb{R},$ we have $U_M(t)\mathcal{Q}=\mathcal{Q}U_H(t).$ Since $\mathcal{Q}\e_a=\u_a,$ we obtain $U_M(t)\u_a=U_M(t)\mathcal{Q}\e_a=\mathcal{Q}U_H(t)\e_a.$ As $\mathcal{Q}^T\mathcal{Q}=I,$ it follows that, for some complex scalars $\alpha, \beta$ with $\beta\neq 0,$ $U_H(t)\e_a=\alpha\e_a+\beta\e_b$ if and only if $U_M(t)\u_a=\alpha\u_a+\beta\u_b.$  
\end{proof}
The following result shows that the characterization of fractional revival in Theorem \ref{6t1} is preserved under a suitable block matrix extension.
\begin{thm}\label{6t4}
    Suppose the premise of Theorem \ref{6t1} holds. Let 
    \[M'=\begin{bmatrix}
        M & B\\
        B^T & F
    \end{bmatrix}\quad\text{with} \quad B^T\u_j=\mathbf{0}, ~~\text{for all}~~1\leq j \leq k,\]
    and $F$ is a real symmetric matrix of appropriate order.
    Then the graph associated with matrix $H$ exhibits fractional revival from vertex $a$ to $b$ if and only if  the graph associated with matrix $M'$ admits fractional revival from $[\u_a,\mathbf{0}]^T$ to $[\u_b,\mathbf{0}]^T.$
\end{thm}
\begin{proof}
Since $B^T\u_j=\mathbf{0}$ for all $1\leq j\leq k$, we have
\[
M'\begin{bmatrix}
\u_j\\
\mathbf{0}
\end{bmatrix}=
\begin{bmatrix}
M & B\\
B^T & F
\end{bmatrix}
\begin{bmatrix}
\u_j\\
\mathbf{0}
\end{bmatrix}
=
\begin{bmatrix}
M\u_j\\
\mathbf{0}
\end{bmatrix}.
\]
Considering $\mathcal{Q'}=\begin{bmatrix}
    \mathcal{Q}\\ \mathbf{0}
\end{bmatrix},$ it follows that
% [\mathcal{Q},\mathbf{0}]^T,$
$
M'\mathcal{Q'}=
\begin{bmatrix}
M\mathcal{Q}\\
\mathbf{0}
\end{bmatrix}
=
\begin{bmatrix}
\mathcal{Q}H\\
\mathbf{0}
\end{bmatrix}
=\mathcal{Q'}H.$
Hence, the desired conclusion follows by applying Theorem~\ref{6t1} to $M',$ $\mathcal{Q'},$ and $H.$ 
\end{proof}

\begin{cor}\label{6c2}
Suppose the premise of Theorem~\ref{6t1} holds for $M\in\{A,L,Q\}$ associated with a graph $G.$ Define 
\[S=\{v\in V(G) \mid (\u_j)_v=0\quad\text{for all}\quad  1\leq j\leq k\}.\]
If $G$ is an induced subgraph of a graph $X$ such that every vertex of $G$ having a neighbour in $V(X)\backslash V(G)$ belongs to $S,$ then 
$G$ admits fractional revival from $\u_a$ to $\u_b$ if and only if the graph $X$ exhibits fractional revival from $[\u_a,\mathbf{0}]^T$ to $[\u_b,\mathbf{0}]^T.$
\end{cor}

\begin{proof}
Let the matrix associated with $X$ be
    \[M'=\begin{bmatrix}
        M+\delta D' & B\\
        B^T & F
    \end{bmatrix},\] where $D'$ is the diagonal matrix defined by $D'_{a,a}=\operatorname{deg}_X(a)-\operatorname{deg}_G(a)$ for every $a\in V(G),$ and $\delta=0$ if $M=A$ and $1$ otherwise. Since $(\u_j)_v=0$ for all $v\in S,$ it follows that $D'\u_j=\mathbf{0}$ and $
B^T\u_j=\mathbf{0},$ for $1\leq j\leq k.$ Therefore, \[M'\begin{bmatrix}
\u_j\\
\mathbf{0}
\end{bmatrix}=
\begin{bmatrix}
M\u_j\\
\mathbf{0}
\end{bmatrix},\]
and the conclusion now follows from Theorem \ref{6t4}.
\end{proof}

\begin{figure}
\centering
%%%%%%%%%%%%%%%%%%%%%%%%%%%%%%%%%%%%%%%%%%%%%%%%%%%%%%%%%%%%%%

\centering
\begin{tikzpicture}[scale=0.7]

\tikzset{
vertex/.style={
circle,
draw,
thick,
fill=white,
minimum size=3.5mm,
inner sep=0pt,
font=\scriptsize
}
}

\node[vertex] (4) at (-1,0) {$4$};
 \node[vertex] (8) at (5,0) {$8$};

\node[vertex] (3) at (-2,1) {$3$};
\node[vertex] (2) at (-3,2) {$2$};
\node[vertex] (1) at (-4,3) {$1$};
\draw[thick] (1)--(2)--(3)--(4);

\node[vertex,minimum size=2.5mm] (12) at (-2,0.4) {};
\node[vertex,minimum size=2.5mm] (13) at (-3,0.8) {};
\node[vertex,minimum size=2.5mm] (14) at (-4,1.2) {};
\draw[thick] (4)--(12)--(13)--(14);

\fill (-3,0.1) circle (1.5pt);
\fill (-3,-0.3) circle (1.5pt);
\fill (-3,-0.7) circle (1.5pt);

\node[vertex,minimum size=2.5mm] (15) at (-2,-0.8) {};
\node[vertex,minimum size=2.5mm] (16) at (-3,-1.6) {};
\node[vertex,minimum size=2.5mm] (17) at (-4,-2.4) {};
\draw[thick] (4)--(15)--(16)--(17);

\node[vertex] (9) at (6,1) {$9$};
\node[vertex] (10) at (7,2) {$10$};
\node[vertex] (11) at (8,3) {$11$};
\draw[thick] (8)--(9)--(10)--(11);

\node[vertex,minimum size=2.5mm] (18) at (6,0.4) {};
\node[vertex,minimum size=2.5mm] (19) at (7,0.8) {};
\node[vertex,minimum size=2.5mm] (20) at (8,1.2) {};
\draw[thick] (8)--(18)--(19)--(20);

\fill (7,0.1) circle (1.5pt);
\fill (7,-0.3) circle (1.5pt);
\fill (7,-0.7) circle (1.5pt);

\node[vertex,minimum size=2.5mm] (21) at (6,-0.8) {};
\node[vertex,minimum size=2.5mm] (22) at (7,-1.6) {};
\node[vertex,minimum size=2.5mm] (23) at (8,-2.4) {};
\draw[thick] (8)--(21)--(22)--(23);

    \node[vertex,minimum size=2.5mm] (24) at (0.5,1) {};
    \node[vertex,minimum size=2.5mm] (26) at (3.5,1) {};
    \draw[thick] (4)--(24);
    \draw[thick] (26)--(8);
\draw[thick, dashed] (24)--(26);

\node[vertex] (27) at (0.5,-1) {$5$};
\node[vertex] (28) at (2,-1){$6$};
\node[vertex] (29) at (3.5,-1) {$7$};
\draw[thick] (4)--(27)--(28)--(29)--(8);

\end{tikzpicture}
\caption{$P_{11}$
with disjoint copies of $P_3$
 attached to vertices $4$ and $8,$ respectively, by an edge incident with an end vertex of each copy, together with an additional path of arbitrary length joining $4$ and $8.$}
 
\label{6fig1}
\end{figure}

\begin{exm}\label{6ex1}
Consider the induced path $P_{11}$ with vertex set $\{j\in\mathbb{N} \mid 1\leq j\leq 11\}$ as shown in Figure \ref{6fig1}, and let $M=A.$ With respect to the vertex ordering $(1,7,9,8,2,6,10,3,5,11,4),$ let $M_1,M_2,$ and $M_3$
denote the principal blocks of orders $4,3,$ and $4,$ respectively as in Lemma \ref{6l1}. Let $\w_1=\w_3=\frac{1}{\sqrt{3}}[1,-1,1,0]^T,$ and $\w_2=\frac{1}{\sqrt{3}}[1,-1,1]^T.$ A direct calculation gives $M_1\w_1=M_2\w_2=M_3\w_3=\mathbf{0},$ $C_{12}\w_2=\w_1,$ $C_{13}\w_3=\mathbf{0},$ and $C_{23}\w_3=\w_2.$ Therefore, all the hypotheses of Lemma \ref{6l1} holds with $\u_1=\frac{1}{\sqrt{3}}\ob{\e_1-\e_7+\e_9},$ $\u_2=\frac{1}{\sqrt{3}}\ob{\e_2-\e_6+\e_{10}},$ and $\u_3=\frac{1}{\sqrt{3}}\ob{\e_3-\e_5+\e_{11}},$ and we obtain 
\[H=\begin{bmatrix}
    0 & 1 & 0\\1 & 0 & 1\\0 & 1 & 0
\end{bmatrix},\]
which is the adjacency matrix of $P_3.$ Since $P_3$ admits PST between the vertices $1$ and $3$ at $\frac{\pi}{\sqrt{2}},$ by Theorem \ref{6t1}, the path $P_{11}$
admits PST between $\frac{1}{\sqrt{3}}\ob{\e_1-\e_7+\e_9}$ and  $\frac{1}{\sqrt{3}}\ob{\e_3-\e_5+\e_{11}}$ at the same time, which is consistent with  \cite[Example 7.3(2)]{god7}. 
Observe that the vectors  $\u_1,\u_2,$ and $\u_3$ have zero entries at vertices $4$ and $8.$ If we add an edge between $4$  and $8$ in $P_{11},$ the same vectors $\w_1,\w_2,$ and $\w_3$ satisfy the hypotheses of Lemma \ref{6l1}, with the same matrix $H.$ Hence, the resulting graph also admits PST between $\frac{1}{\sqrt{3}}\ob{\e_1-\e_7+\e_9}$ and  $\frac{1}{\sqrt{3}}\ob{\e_3-\e_5+\e_{11}}.$

Corollary \ref{6c2} implies that attaching disjoint copies of $P_3$ to vertices $4$ and $8$ by an edge incident with an end vertex of each copy,  preserves PST between
$
\frac{1}{\sqrt{3}}(\e_1-\e_7+\e_9)$
and 
$\frac{1}{\sqrt{3}}(\e_3-\e_5+\e_{11})$
in the resulting graph. Moreover, it follows from Corollary \ref{6c2} that the graph in Figure \ref{6fig1} exhibits PST between $\frac{1}{\sqrt{3}}\ob{\e_1-\e_7+\e_9}$ and  $\frac{1}{\sqrt{3}}\ob{\e_3-\e_5+\e_{11}}.$
\end{exm}
The observations in Example \ref{6ex1} lead to the following result, where an $(m,L)$-state in a graph is a linear combination $l_1\e_{j_1}+l_2\e_{j_2}+\cdots +l_m\e_{j_m}$ of $m$ vertex states, with $L=(l_1,l_2,\ldots,l_m)\in \mathbb{C}^m$ satisfying $\sum_{k=1}^m|l_k|^2=1$ \cite[Section 5]{pal10}.
\begin{cor}\label{6cor3}
   let $a$ and $b$ be two vertices in a cycle  $C_n$ with $n\geq5,$  that are joined by a path of length four. Suppose that two disjoint copies of $P_3$ are connected to $a$ and $b,$ respectively, by an edge joining each of $a$ and $b$ to an end vertex of the corresponding copy of $P_3.$ Then the resulting graph exhibits perfect $(m,L)$-state transfer with respect to the adjacency matrix where $m=3$ and $L=\ob{\frac{1}{\sqrt{3}},-\frac{1}{\sqrt{3}},\frac{1}{\sqrt{3}}}.$
\end{cor}
The constructions in Example \ref{6ex1} and Corollary \ref{6cor3} provide infinitely many unweighted graphs with perfect $(m,L)$-state transfer.
\begin{cor}
    For each $k\geq 3,$ there are infinitely many connected unweighted graphs (resp., trees, unicyclic graphs) with maximum degree $k$ admitting perfect $(m,L)$-state transfer with respect to the adjacency matrix where $m=3$ and $L=\ob{\frac{1}{\sqrt{3}},-\frac{1}{\sqrt{3}},\frac{1}{\sqrt{3}}}.$
\end{cor}

\section{Fractional revival in complementary prisms}\label{6s4}
Throughout this section, we assume that the graph $G$ is simple and connected.
\begin{thm}\label{6t5}
Let $M\in\{A,L,Q\}$, and
$\u$ be a fixed state in $G$ relative to $M$ satisfying
$\mathbf{1}^T\u=0$. Then, for every $t\in\mathbb{R}$,
\[
U_{M(G\overline{G})}(t)
[\u,
\mathbf{0}]^T
=
\alpha(t)[\u,
\mathbf{0}]^T+
\beta(t)
[\mathbf{0},\u]^T,
\]
for some complex-valued functions $\alpha(t)$ and $\beta(t)$.

\end{thm}

\begin{proof}
Since $\u$ is a fixed state in $G$ relative to $M,$ there exists $\lambda\in\mathbb{R}$
such that $M(G)\u=\lambda\u.$
Further, $\mathbf1^T\u=0$ implies $J\u=\mathbf{0}.$  Consequently, $\u$ is also an eigenvector of $M(\overline G)$, say
$M(\overline G)\u=\mu\u.$
Now in $M(G\overline{G}),$ consider
\[
M_1=M(G)+\delta I,\quad
M_2=M(\overline G)+\delta I,\quad C_{12}=\zeta I.\] 
For $\w_1=\w_2=\u,$ all the hypotheses of Lemma \ref{6l1} are satisfied, and the matrix $H$ is given by
\[
H=
\begin{bmatrix}
\lambda+\delta & \zeta\\
\zeta & \mu+\delta
\end{bmatrix}.
\]
Since $U_H(t)\e_1=\alpha(t)\e_1+\beta(t)\e_2$ for some complex-valued functions $\alpha(t)$ and $\beta(t),$ by using Theorem \ref{6t1}, we obtain
\[U_{M(G\overline{G})}(t)[\u,
\mathbf{0}]^T
=
\alpha(t)[\u,
\mathbf{0}]^T+
\beta(t)
[\mathbf{0},\u]^T.
\]
\end{proof}

We next characterize the existence of PST in the complementary prism of a graph.
\begin{thm}\label{6t7} 
Suppose the premise of Theorem \ref{6t5} holds and the eigenvalues corresponding to eigenvector $\u$ of $M(G)$ and $M(\overline{G})$ are $\lambda$ and $\mu,$ respectively.
Then the complementary prism $G\overline{G}$ exhibits perfect state transfer between
$[\u,
\mathbf{0}]^T$ and
$[\mathbf{0},\u]^T$
if and only if $\lambda=\mu$. Moreover, perfect state transfer occurs at $
t=\frac{(2k+1)\pi}{2}$ for some $k\in \mathbb{Z}.$
\end{thm}

\begin{proof}
    By Theorem \ref{6t1}, the complementary prism $G\overline{G}$ exhibits PST between $[\u,
\mathbf{0}]^T$ and
$[\mathbf{0},\u]^T$ if and only if the graph associated with the matrix $H$ given in the proof of Theorem \ref{6t5} admits PST between its two vertices. If the graph associated with $H$ admits PST between its two vertices, then $H$ must be mirror symmetric \cite[Lemma 2]{kay}, and hence $\lambda=\mu.$ Conversely, if $\lambda=\mu,$ then 
\[
U_H(t)
=
\exp{\ob{it\ob{\lambda+\delta}}}
\begin{bmatrix}
\cos(\zeta t) & i\sin(\zeta t)
\\
i\sin(\zeta t)
& \cos(\zeta t)
\end{bmatrix}.
\]
Since $\zeta=\pm 1,$ the graph associated with $H$ admits PST if and only if $t=\frac{(2k+1)\pi}{2},$ for some $k\in\mathbb{Z}.$
\end{proof}

The following result identifies a pair of states that exhibit Laplacian PST in the complementary prism of every graph.
\begin{thm}\label{5t4}
     Let $G$ be a graph on $n$ vertices. The complementary prism of $G$ exhibits perfect state transfer between $\frac{1}{\sqrt{n}}[\mathbf{1},\mathbf{0}]^T$ and $\frac{1}{\sqrt{n}}[\mathbf{0},\mathbf{1}]^T$ at an odd multiple of $\frac{\pi}{2}$ with respect to the Laplacian matrix.
\end{thm}
\begin{proof}
    For the Laplacian matrix of a complementary prism of G, we have $M_1=L(G)+I,$ $M_2=L(\overline{G})+I,$ and $C_{12}=C_{12}^T=-I.$ Since $M_1\ob{\frac{1}{\sqrt{n}}}\mathbf{1}=M_2\ob{\frac{1}{\sqrt{n}}}\mathbf{1}=\ob{\frac{1}{\sqrt{n}}}\mathbf{1},$ using Lemma \ref{6l1}, we obtain \[H=\begin{bmatrix}
    1 & -1\\
    -1 & 1
    \end{bmatrix},
    \] which represents the Laplacian matrix of $P_2.$ Since $P_2$ exhibits PST between its end vertices at an odd multiple of $\frac{\pi}{2},$  Theorem \ref{6t1} implies that the complementary prism also exhibits PST between $\frac{1}{\sqrt{n}}[\mathbf{1},\mathbf{0}]^T$ and $\frac{1}{\sqrt{n}}[\mathbf{0},\mathbf{1}]^T$ at the same time. 
\end{proof}

\begin{exm}
Let $P_n$ be a path on 
$n$ vertices, with vertex set 
$\{1,2,\ldots,n\}$ labeled consecutively along the path. One may observe $P_4$ as a complementary prism of a complete graph on two vertices. Then Theorem \ref{5t4} yields Laplacian PST between $\frac{1}{\sqrt{2}}(\e_1+\e_4)$ and $\frac{1}{\sqrt{2}}(\e_2+\e_3)$ at $\frac{\pi}{2},$ which is consistent with \cite[Remark 7.9(2)]{god7}. 
\end{exm}
 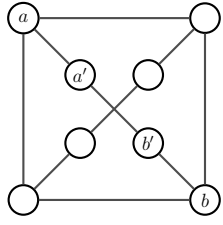
\begin{figure}
\begin{center}
\begin{tikzpicture}[scale=.3,auto=left]
                       \tikzstyle{every node}=[circle, thick, fill=white, scale=0.6]
                       
		        \node[draw] (1) at (-4,4) {$a$};		        
		        \node[draw,minimum size=0.65cm, inner sep=0 pt] (2) at (-4, -4) {};
		        \node[draw,minimum size=0.65cm, inner sep=0 pt] (3) at (4, -4) {$b$};
		        \node[draw,minimum size=0.65cm, inner sep=0 pt] (4) at (4, 4) {};
		        \node[draw,minimum size=0.65cm, inner sep=0 pt] (5) at (-1.5, 1.5) {$a'$};		 \node[draw,minimum size=0.65cm, inner sep=0 pt] (6) at (-1.5, -1.5) {};       
	\node[draw,minimum size=0.65cm, inner sep=0 pt] (7) at (1.5, -1.5) {$b'$};
    \node[draw,minimum size=0.65cm, inner sep=0 pt] (8) at (1.5, 1.5) {};

				\draw [thick, black!70] (1)--(2)--(3)--(4)--(1);
				
                \draw [thick, black!70] (1)--(5)--(7)--(3);
\draw [thick, black!70] (2)--(6)--(8)--(4);
		  
				\end{tikzpicture}
				
\end{center}	
\caption{\label{5fig1} The complementary prism $C_4\overline{C_4}.$}
\end{figure}
\begin{exm}
Let $a$ and $b$ be a pair of antipodal vertices in $C_4,$ a cycle of length four. Then the state $\frac{1}{\sqrt{2}}(\e_a-\e_b)$ becomes a fixed state in $C_4.$ Therefore, by Theorem \ref{6t5} the complementary prism $C_4\overline{C_4}$ exhibits fractional revival from $\u_1=\frac{1}{\sqrt{2}}(\e_a-\e_b)$ to $\u_2=\frac{1}{\sqrt{2}}(\e_{a'}-\e_{b'})$ [see Figure \ref{5fig1}] with respect to the adjacency, Laplacian and signless Laplacian matrices. In particular, at $t=\frac{\pi}{2},$ $C_4\overline{C_4}$ exhibits PST between $\frac{1}{\sqrt{2}}(\e_a-\e_b)$ and $\frac{1}{\sqrt{2}}(\e_{a'}-\e_{b'})$ with respect to the Laplacian matrix. Moreover, let $X$ be the graph obtained from $C_4\overline{C_4}$ by adjoining two vertices $u$ and $v,$ which are adjacent only to all the vertices of the induced subgraphs $C_4$ and $\overline{C_4}$ of $C_4\overline{C_4},$ respectively. Now, the adjacency, Laplacian and signless Laplacian matrices associated with the graph $X$ can be realized as the block matrix $M'$ appearing in Theorem \ref{6t4}. Since $\mathbf{1}^T\u_j=0,$ for $j=1,2,$ it follows that $B^T\u_j=\mathbf{0}.$ Hence, by Theorem \ref{6t4} the graph $X$ exhibits fractional revival from $[\u_1,\mathbf{0}]^T$ to $[\u_2,\mathbf{0}]^T$ with respect to the adjacency, Laplacian, and signless Laplacian matrices. In particular at $t=\frac{\pi}{2},$ the graph $X$ exhibits PST between $[\u_1,\mathbf{0}]^T$ and $[\u_2,\mathbf{0}]^T$ with respect to the Laplacian matrix. 
\end{exm}
Two vertices $a$ and $b$ in $G$ are said to be twins if $N_G(a)\backslash\{b\}=N_G(b)\backslash\{a\},$ where $N_G(a)$ denotes the set of vertices that are adjacent to vertex $a$ in $G.$ If $a$ and $b$ are adjacent, then they are called true twins; otherwise, they are called false twins. In either case, the state $\frac{1}{\sqrt{2}}(\e_a-\e_b)$ is a fixed state in $G$ with respect to the adjacency, Laplacian, and signless Laplacian matrices \cite[Example 2.4(iii)]{god7}. Therefore, applying Theorem \ref{6t5}, we obtain the following result.
\begin{cor}\label{6c3}
  Let $a$ and $b$ be twin vertices in a graph $G.$  Then the complementary prism $G\overline{G}$ exhibits fractional revival from $\frac{1}{\sqrt{2}}(\e_a-\e_b)$ to $\frac{1}{\sqrt{2}}(\e_{a'}-\e_{b'})$ with respect to the adjacency, Laplacian and signless Laplacian matrices.   
\end{cor}
\begin{exm}\label{5ex3}
    Let $K_n$ be a complete graph on $n$ vertices. The corona product of $K_n$ and $K_1$ is the complementary prism $K_n\overline{K_n}.$ Since any two distinct vertices $a$ and $b$ in $K_n$ are twins, the state $\frac{1}{\sqrt{2}}(\e_a-\e_b)$ is a fixed state in $K_n.$ Therefore, Corollary \ref{6c3} implies that the corona product of $K_n$ and $K_1$ exhibits fractional revival from the state $\frac{1}{\sqrt{2}}(\e_a-\e_b)$ to $\frac{1}{\sqrt{2}}(\e_{a'}-\e_{b'})$ with respect to the adjacency, Laplacian, and signless Laplacian matrices.
\end{exm}
\begin{exm}
  Let $K_{m,n}$ be a complete bipartite graph, and $a$ and $b$ be any two vertices in the same partite set. Then $\frac{1}{\sqrt{2}}(\e_a-\e_b)$ becomes a fixed state in $K_{m,n}.$ Hence, by Corollary \ref{6c3}, the complementary prism $K_{m,n}\overline{K_{m,n}}$  exhibits fractional revival from the state $\frac{1}{\sqrt{2}}(\e_a-\e_b)$ to $\frac{1}{\sqrt{2}}(\e_{a'}-\e_{b'})$ with respect to the adjacency, Laplacian, and signless Laplacian matrices.
\end{exm}

\begin{rem}
   Let $a$ and $b$ be twin vertices in a graph $G.$ Then
$\frac{1}{\sqrt{2}}(\e_a-\e_b)$ is a fixed state in $G$ with respect to $M\in \{A,L,Q\}.$ The existence of PST between $
\frac{1}{\sqrt{2}}(\e_a-\e_b)$
and $\frac{1}{\sqrt{2}}(\e_{a'}-\e_{b'})$ in the complementary prism $G\overline{G}$
can be observed using \cite[Corollary 4]{god8}  and \cite[Corollary 1(1(a))]{ojha1} for $M=A$ and $M\in\{L,Q\},$ respectively, where the subgraphs induced by $\{a,a'\}$ and $\{b,b'\}$ form edge-perturbed twin subgraphs, while the subgraph induced by $\{a,a'\}$ can be regarded as a half graph in $G\overline{G}.$  
 
\end{rem}

Further, we investigate the existence of PST in the complementary prisms of the complete graph and the complete bipartite graph. 
\subsection{Complementary prism $K_n\overline{K_n}$}
 The  eigenvalues of $A(K_n$) are $n-1$ and $-1$ with multiplicities $1$ and $n-1,$ respectively. The corresponding orthonormal eigenvectors are $u_1=\frac{1}{\sqrt{n}}\tb{1,1,\ldots,1}^T,$ for the eigenvalue $n-1,$ and 
 \[
\u_{k+1} = \frac{1}{\sqrt{k(k+1)}}[\underbrace{1,1,\dots,1}_{k-\text{times}},\,-k,\,0,\dots,0]^T, \quad\text{where $k=1,2,\ldots,n-1,$}\]
for the eigenvalue $-1.$
Using Theorem \ref{5cor1}, the eigenvalues and the corresponding orthonormal eigenvectors of $A\ob{K_n\overline{K_n}}$ are as follows.  
\begin{enumerate}[label=(\roman*)]
\item $\alpha_{\pm}(-1)=\dfrac{-1\pm\sqrt{5}}{2},$ each with multiplicity $n-1,$ and the corresponding orthonormal eigenvectors are $\frac{2}{\sqrt{10\pm2\sqrt{5}}}\tb{\mathbf{u}_j,
    \ob{\frac{1\pm\sqrt{5}}{2}}\mathbf{u}_j}^T,$ for $2\leq j\leq n,$ which are orthogonal to the all one vector.
 \item $\beta_{\pm}(n,n-1)=\dfrac{n-1\pm\sqrt{(n-1)^2+4}}{2}$ with eigenvectors $\frac{1}{\sqrt{{n\ob{1+ c_{\pm}^2}}}}[\mathbf{1},c_{\pm}\mathbf{1}]^T,$ where $c_{\pm}=\beta_{\pm}-(n-1).$ 
\end{enumerate}
For $n\geq 2,$ the eigenvalues $\alpha_{\pm}\neq \beta_{\pm}$ and $c_{\pm}\neq 1.$

\begin{thm}\label{6t9}
    There is no perfect pair state transfer in the complementary prism $K_n\overline{K_n}$ with respect to the adjacency matrix for $n\geq 2.$
\end{thm}
\begin{proof}
    Let $\frac{1}{\sqrt{2}}\ob{\e_a-\e_{b'}}$ be a pair state in $K_n\overline{K_n},$ where $a\in V(K_n)$ and $b'\in V(\overline{K_n}).$ The eigenvalue support of $\frac{1}{\sqrt{2}}\ob{\e_a-\e_{b'}}$ contains $\alpha_{\pm}(-1)$ and $\beta_{\pm}(n,n-1).$ Hence by Theorem \ref{6t3}, the state is not periodic, resulting in no pair PST from $\frac{1}{\sqrt{2}}\ob{\e_a-\e_{b'}}.$ 
The eigenvalue supports of $\frac{1}{\sqrt{2}}\ob{\e_a-\e_b}$ and $\frac{1}{\sqrt{2}}\ob{\e_{a'}-\e_{b'}}$ contain $\alpha_{\pm}(-1).$ Using Proposition \ref{6p2}, no pair state of the form $\frac{1}{\sqrt{2}}\ob{\e_c-\e_d}$ and $\frac{1}{\sqrt{2}}\ob{\e_{c'}-\e_{d'}}$ are strongly cospectral with $\frac{1}{\sqrt{2}}\ob{\e_a-\e_b}$ or $\frac{1}{\sqrt{2}}\ob{\e_{a'}-\e_{b'}}.$ Therefore, using \cite[Lemma 5.1]{god7}, there is no pair PST in the complementary prism of a complete graph. 
\end{proof}
Recall that the vertex corona $K_n\circ K_1$ can be realized as the complementary prism $K_n\overline{K_n}.$ Wang et al. \cite{wang} provided sufficient conditions for the existence and non-existence of Laplacian pair PST in vertex coronas. In particular, \cite[Example 3.2]{wang} shows the non-existence of Laplacian pair PST between certain types of pair states in $K_n\circ H$, where $n\geq 3$ and $H$ is a connected graph on $m$ vertices. The following result provides a complete characterization of Laplacian pair PST in $K_n\overline{K_n}$.

\begin{thm}\label{6th10}
   The complementary prism  $K_n\overline{K_n}$ admits perfect pair state transfer with respect to the Laplacian matrix if and only if $n=2.$ 
\end{thm}
\begin{proof}
If $n=2,$ the  complementary prism $K_2\overline{K_2}$ represents the path $P_4,$ which admits Laplacian pair PST at $\frac{\pi}{\sqrt{2}}$ \cite[Theorem 7.7]{che1}. Now suppose that $n\neq 2.$ 
    Since $K_n$ is regular, the vectors $\u_1$ and $\u_j$ for $2\leq j\leq n,$ defined as eigenvectors of $A(K_n),$ are also eigenvectors of $L(K_n)$ corresponding to the eigenvalues $0$ and $n,$ respectively.
    The eigenvalues of $L\ob{K_{n}\overline{K_{n}}}$ are $2, 0$ and $\alpha_{\pm}(n)=\frac{n+2\pm\sqrt{n^2+4}}{2},$ with eigenvectors as given in Theorem \ref{6th2}. Let $a,b$ and $a',b'$ be vertices in $K_{n}\overline{K_{n}},$ where $a,b\in V(K_{n})$ and $a',b'\in V(\overline{K_{n}}).$ The eigenvalue support of $\frac{1}{\sqrt{2}}\ob{\e_a-\e_{b'}}$ contains $2$ and $\alpha_{\pm}(n).$  Hence Theorem \ref{6t3} implies that the state is not periodic, resulting in no pair PST from $\frac{1}{\sqrt{2}}\ob{\e_a-\e_{b'}}.$ 
The eigenvalue supports of $\frac{1}{\sqrt{2}}\ob{\e_a-\e_b}$ and $\frac{1}{\sqrt{2}}\ob{\e_{a'}-\e_{b'}}$ contain $\alpha_{\pm}(n).$ Using Proposition \ref{6p2}, no pair state of the form $\frac{1}{\sqrt{2}}\ob{\e_c-\e_d}$ and $\frac{1}{\sqrt{2}}\ob{\e_{c'}-\e_{d'}}$ are strongly cospectral with $\frac{1}{\sqrt{2}}\ob{\e_a-\e_b}$ or $\frac{1}{\sqrt{2}}\ob{\e_{a'}-\e_{b'}}.$ Therefore, by \cite[Lemma 5.1]{god7}, there is no pair PST in $K_n\overline{K_n}$ for $n\neq 2.$ 
\end{proof}
Since the complete graph $K_n$ is regular, the vectors $\u_1$ and $\u_j$ for $2\leq j\leq n,$ defined as eigenvectors of $A(K_n),$ are also eigenvectors of $Q(K_n)$ corresponding to the eigenvalues $2n-2$ and $n-2,$ respectively. The eigenvalues of $Q\ob{K_{n}\overline{K_{n}}}$ are $\alpha_{\pm}(n-2)=\frac{n\pm\sqrt{(n-2)^2+4}}{2}$ and $\beta_{\pm}(n,n-1)=n\pm\sqrt{(n-1)^2+1},$ with the corresponding eigenvectors as given in Theorem \ref{5t5}. 
For $n=2,$ let $\u=\frac{1}{\sqrt{2}}[1,-1]^T.$ Since $Q(K_2)\u=Q(\overline{K_2})\u=\mathbf{0},$ Theorem \ref{6t7} implies that $K_2\overline{K_2}$ exhibits PST between $[\u,\mathbf{0}]^T$ and $[\mathbf{0},\u]^T$ at $\frac{\pi}{2}.$ Since $K_2\overline{K_2}$ is isomorphic to $P_4,$ the path $P_4$ admits PST between  $\frac{1}{\sqrt{2}}\ob{\e_1-\e_4}$ and $\frac{1}{\sqrt{2}}\ob{\e_2-\e_3}$ at $\frac{\pi}{2}$ with respect to the signless Laplacian matrix. 
For $n>2,$ using an argument analogous to that in the proof of Theorem \ref{6t9}, we obtain the following result.

\begin{thm}
    The complementary prism  $K_n\overline{K_n}$ admits perfect pair state transfer with respect to the signless Laplacian matrix if and only if $n=2.$
 \end{thm}

\subsection{Complementary prism $K_{m,n}\overline{K_{m,n}}$}
 We investigate the existence of pair PST and plus PST in the complementary prism of a complete bipartite graph.
\begin{thm}\label{6t10}
Let $K_{m,n}$ be a complete bipartite graph. If $a$ and $b$ are two vertices in a partite set having $m$ vertices, then the complementary prism $K_{m,n}\overline{K_{m,n}}$ exhibits perfect state transfer between $\frac{1}{\sqrt{2}}(\e_a-\e_b)$ and $\frac{1}{\sqrt{2}}(\e_{a'}-\e_{b'})$ with respect to the Laplacian and signless Laplacian matrices if and only if $m=n$ and $m-n=2,$ respectively. In the case of the adjacency matrix, there is no perfect pair state transfer between $\frac{1}{\sqrt{2}}(\e_a-\e_b)$ and $\frac{1}{\sqrt{2}}(\e_{a'}-\e_{b'})$.
\end{thm}
\begin{proof}
    Let $\u=\frac{1}{\sqrt{2}}(\e_a-\e_b).$ One may observe that $\u$ is a fixed state in $K_{m,n}$ and $\overline {K_{m,n}}$ with
$\mathbf1^T\mathbf u=0.$ Therefore, for some $\lambda$ and $\mu,$ we have 
$M(K_{m,n})\mathbf u=\lambda\u$, $M(\overline{K_{m,n}})\u=\mu\u,$ where
\[
(\lambda,\mu) =
\begin{cases}
(0,-1), & \text{if}\quad M=A,\\
(n,m), & \text{if}\quad M=L,\\
(n,m-2), & \text{if}\quad M=Q.
\end{cases}
\]
By Theorem \ref{6t7}, PST occurs between $\frac{1}{\sqrt{2}}(\e_a-\e_b)$ and $\frac{1}{\sqrt{2}}(\e_{a'}-\e_{b'})$ if and only if $\lambda=\mu,$ and hence the result follows. 
\end{proof}
Theorem \ref{6t10} provides an infinite family of graphs having PST between pair states corresponding to an edge and a non-edge.
The eigenvalues and corresponding orthonormal eigenvectors of the complete bipartite graph $K_{m,n}$ with respect to the Laplacian matrix are as follows.

\begin{enumerate}[label=(\roman*)]
    \item Eigenvalue $0$ with eigenvector $\frac{1}{\sqrt{m+n}}\mathbf{1}.$
    \item Eigenvalue $m+n$ with eigenvector $\frac{1}{\sqrt{mn(m+n)}}[
            n\mathbf{1}_m,
        -m\mathbf{1}_n
        ]^T.$
        \item Eigenvalue $n$ with algebraic multiplicity $m-1$ and eigenvectors $\u_k=[
           \f_k,
        \mathbf{0}]^T,$ where $\f_k=\frac{1}{\sqrt {k(k+1)}}[\underbrace{1,1,\dots,1}_{k-\text{times}},-k,0,\ldots,0]^T.$
        \item Eigenvalue $m$ with algebraic multiplicity $n-1$ and eigenvectors $\v_k=[
        \mathbf{0},\g_k]^T,$ where $\g_k=\frac{1}{\sqrt {k(k+1)}}[\underbrace{1,1,\dots,1}_{k-\text{times}},-k,0,\ldots,0]^T.$ 
\end{enumerate}

The following result completely characterizes the existence of pair PST in $K_{m,n}\overline{K_{m,n}}$ with respect to the Laplacian matrix. 
\begin{thm}
The complementary prism of a complete bipartite graph $K_{m,n}$ exhibits Laplacian perfect pair state transfer if and only if $m=n.$
\end{thm}
\begin{proof}
Suppose $m\neq n.$
    The eigenvalues of $L\ob{K_{m,n}\overline{K_{m,n}}}$ are $2,$ $0,$ $\alpha_{\pm}(m+n),$ $\alpha_{\pm}(n),$ and $\alpha_{\pm}(m),$ where
    % $\alpha_{1,2}(m+n),$ $\alpha_{1,2}(n)$, and $\alpha_{1,2}(m),$ where
    \[\alpha_{\pm}(m+n)=\frac{m+n+2\pm\sqrt{(m+n)^2+4}}{2},\quad\alpha_{\pm}(n)=\alpha_{\pm}(m)=\frac{m+n+2\pm\sqrt{(m-n)^2+4}}{2}\] with eigenvectors as given in Theorem \ref{6th2}. Let $a,b$ and $a',b'$ be vertices in $K_{m,n}\overline{K_{m,n}},$ where $a,b\in V(K_{m,n})$ and $a',b'\in V(\overline{K_{m,n}}).$ For any two vertices $a$ and $b$ in $K_{m,n},$ one may observe that $2,\alpha_{\pm}(m+n)\in\sigma_{\frac{1}{\sqrt{2}}\ob{\e_a-\e_{b'}}},$ while if   $a$ and $b$ are in the different partite set of $K_{m,n},$ then  the eigenvalue support of $\frac{1}{\sqrt{2}}\ob{\e_a-\e_b}$ and $\frac{1}{\sqrt{2}}\ob{\e_{a'}-\e_{b'}}$ contain $\alpha_{\pm}(m+n)$ and $\alpha_{\pm}(n).$ Hence by Theorem \ref{6t3}, none of these states is periodic, resulting in no pair PST from these states. 
    If $a$ and $b$ are in the same partite set of $K_{m,n},$ then the eigenvalue supports of  $\frac{1}{\sqrt{2}}\ob{\e_a-\e_b}$ and $\frac{1}{\sqrt{2}}\ob{\e_{a'}-\e_{b'}}$ contain $\alpha_{\pm}(n).$ Moreover, if vertices $c$ and $d$ are in same partite set of $K_{m,n}$ with $\{a,b\}\neq \{c,d\},$ then Proposition \ref{6p2} shows that neither $\frac{1}{\sqrt{2}}\ob{\e_c-\e_d}$ nor $\frac{1}{\sqrt{2}}\ob{\e_{c'}-\e_{d'}}$ is strongly cospectral with $\frac{1}{\sqrt{2}}\ob{\e_a-\e_b}$ or $\frac{1}{\sqrt{2}}\ob{\e_{a'}-\e_{b'}}.$ Therefore, by \cite[Lemma 5.1]{god7}, there is no pair PST in $K_{m,n}\overline{K_{m,n}}.$  
    The sufficient part follows directly from Theorem \ref{6t10}.
\end{proof}
Next, we investigate the existence of plus PST in the complementary prism $K_{n,n}\overline{K_{n,n}}.$ 
Using Theorem \ref{6th2}, the Laplacian eigenvalues and corresponding orthonormal eigenvectors of the complementary prism $K_{n,n}\overline{K_{n,n}}$ are as follows:
\begin{enumerate}[label=(\roman*)]
    \item Eigenvalue $0$ with eigenvector $\frac{1}{2\sqrt{n}}\mathbf{1}.$
    \item Eigenvalue $2$ with eigenvector $\frac{1}{2\sqrt{n}}[\mathbf{1},
        -\mathbf{1}]^T.$
    \item Eigenvalue $n+1\pm\sqrt{n^2+1}$ with eigenvector $\frac{1}{\sqrt{1+d_{\pm}^2}}[
            \w,
       d_{\pm}\w]^T,$ where $d_{\pm}=n\mp\sqrt{n^2+1}$ and $\w=\frac{1}{\sqrt{2n}}[
            \mathbf{1}_n,-\mathbf{1}_n]^T.$
        \item Eigenvalue $n+2$ with algebraic multiplicity $2n-2$ and eigenvectors $\frac{1}{\sqrt{2}}
           [\u_k,
        -\u_k]^T,$ and $\frac{1}{\sqrt{2}}
           [\v_k,
        -\v_k]^T,$ where $\u_k$ and $\v_k$ are as defined for $L(K_{m,n}).$
        \item Eigenvalue $n$ with algebraic multiplicity $2n-2$ and eigenvectors $\frac{1}{\sqrt{2}}[\u_k,
        \u_k]^T,$ and $\frac{1}{\sqrt{2}}[
           \v_k,
        \v_k]^T,$ where $\u_k$ and $\v_k$ are as defined earlier.
\end{enumerate}
The following result provides an infinite family of graphs exhibiting plus PST.

\begin{thm}
      Let $a$ and $b$ be two vertices in the different partite set of a complete bipartite graph $K_{n,n}$ with $n>2.$ Then the complementary prism $K_{n,n}\overline{K_{n,n}}$ exhibits Laplacian perfect state transfer between $\frac{1}{\sqrt{2}}(\e_a+\e_b)$ and $\frac{1}{\sqrt{2}}(\e_{a'}+\e_{b'})$ if and only if $n$ is a multiple of $4.$ If $n=2,$ then $K_{2,2}\overline{K_{2,2}}$ exhibits perfect state transfer between $\frac{1}{\sqrt{2}}\ob{\e_1+\e_3}$ and $\frac{1}{\sqrt{2}}\ob{\e_{2'}+\e_{4'}}.$
\end{thm}
\begin{proof}
Let $\u=\frac{1}{\sqrt{2}}\ob{\e_a+\e_b}$ and $\v=\frac{1}{\sqrt{2}}\ob{\e_{a'}+\e_{b'}}.$ The eigenvalue support of $\u$ contains $0,2,n+2$ and $n.$ One may observe that  $0,n\in\sigma_{\u,\v}^+(L)$ and $2,n+2\in \sigma_{\u,\v}^-(L)$ for $n>2.$ Let $\lambda_1=n.$ By \cite[Corollary 5.7(3)]{god7}, PST occurs between $\u$ and $\v,$ if and only if 
    \[\nu_2(n)>\nu_2(n-2)=\nu_2(-2)=1.\] 
    Now, $\nu_2(n-2)=1$ if and only if $n\equiv 0 \pmod {4}.$
    
For $n=2,$ the states $\u=\frac{1}{\sqrt{2}}\ob{\e_1+\e_3}$ and $\w=\frac{1}{\sqrt{2}}\ob{\e_{2'}+\e_{4'}}$ are strongly cospectral with $0,4\in \sigma_{\u,\w}^+(L)$ and $2\in \sigma_{\u,\w}^-(L).$ Now consider $\lambda_1=4.$ Since $v_2(4)>v_2(2),$  by \cite[Corollary 5.7(3)]{god7}, we have PST between $\frac{1}{\sqrt{2}}\ob{\e_1+\e_3}$ and $\frac{1}{\sqrt{2}}\ob{\e_{2'}+\e_{4'}}.$
\end{proof}
The adjacency matrix of $K_{n,n}$ has eigenvalues $n,-n,$ and $0.$ The orthonormal eigenvectors corresponding to $n$ and $-n$ are $\frac{1}{\sqrt{2n}}[\mathbf{1}_n, \mathbf{1}_n]^T$ and $\frac{1}{\sqrt{2n}}[\mathbf{1}_n, -\mathbf{1}_n]^T,$ respectively. For the eigenvalue $0,$ the orthonormal eigenvectors are  $[\w_j,\mathbf{0}]^T$ and  $[\mathbf{0},\w_j]^T,$  $1\leq j \leq n-1,$ where  $\w_j=\frac{1}{\sqrt {j(j+1)}}\tb{1,1,\ldots,1,-j,0,\ldots,0}^T.$
Using Theorem \ref{6t3}, 
The following result rules out the existence of plus PST in the complementary prism $K_{n,n}\overline{K_{n,n}}.$
\begin{thm}\label{6t13}
    There is no perfect plus state transfer in the complementary prism of a complete bipartite graph $K_{n,n}$ with respect to the adjacency matrix.
\end{thm}
\begin{proof}
For $n=1,$ the complementary prism $K_{n,n}\overline{K_{n,n}}$ is isomorphic to $P_4,$ and $P_4$ does not admit plus PST relative to $A$ \cite[Corollary 7.5]{god7}. For $n\geq 2,$
   the eigenvalues of $K_{n,n}\overline{K_{n,n}}$ are $\beta_{\pm}(2n,n)=\frac{2n-1\pm\sqrt{5}}{2},$ $\alpha_{\pm}(-n)=\frac{-1\pm\sqrt{(1-2n)^2+4}}{2},$ and $\alpha_{\pm}(0)=\frac{-1\pm\sqrt{5}}{2},$ with eigenvectors as given in Theorem \ref{5cor1}. Let $a,b$ and $a',b'$ be the vertices in $K_{n,n}$ and $\overline{K_{n,n}},$ respectively. If $a$ and $b$ are in the different partite set of $K_{n,n},$ then the eigenvalue support of $\frac{1}{\sqrt{2}}\ob{\e_a+\e_b}$ and $\frac{1}{\sqrt{2}}\ob{\e_{a'}+\e_{b'}}$ contain $\beta_{\pm}(2n,n)$ and $\alpha_{\pm}(0).$ If $a$ and $b$ are in the same partite set, then the eigenvalue supports of plus states $\frac{1}{\sqrt{2}}\ob{\e_a+\e_b}$ and $\frac{1}{\sqrt{2}}\ob{\e_{a'}+\e_{b'}}$ contain $\beta_{\pm}(2n,n)$ and $\alpha_{\pm}(-n).$ Lastly, for any $a,b$ in $K_{n,n},$ the eigenvalue support of $\frac{1}{\sqrt{2}}\ob{\e_a+\e_{b'}}$ contains $\beta_{\pm}(2n,n)$ and $\alpha_{\pm}(-n).$ Therefore none of these states are periodic by Theorem \ref{6t3}, resulting in no plus PST in $K_{n,n}\overline{K_{n,n}}.$
\end{proof}
Since $K_{n,n}$ is $n$-regular and $Q(K_{n,n})=nI+A(K_{n,n}),$ the orthonormal eigenvectors of  $A(K_{n,n})$ are also orthonormal eigenvectors of $Q(K_{n,n})$ with corresponding eigenvalues $2n,0,$ and $n,$ respectively. Therefore, the eigenvalues of $Q\ob{K_{n,n}\overline{K_{n,n}}}$ are $\beta_{\pm}(2n,n)=2n\pm\sqrt{2},$ $\alpha_{\pm}(0)=n\pm\sqrt{(n-1)^2+1},$ and $\alpha_{\pm}(n)=n\pm\sqrt{2},$ with eigenvectors as given in Theorem \ref{5t5}. For $n=1,$ the complementary prism $K_{n,n}\overline{K_{n,n}}$ represents the path $P_4.$ Since $P_4$ admits pair PST between $\frac{1}{\sqrt{2}}(\e_1-\e_2)$ and $\frac{1}{\sqrt{2}}(\e_3-\e_4)$ with respect to the Laplacian matrix,  \cite[Theorem 8.3]{che1} implies that $P_4$ also admits PST between $\frac{1}{\sqrt{2}}(\e_1+\e_2)$ and $\frac{1}{\sqrt{2}}(\e_3+\e_4)$ with respect to the signless Laplacian matrix.  For $n\geq 2,$
using a similar approach as in Theorem \ref{6t13}, we obtain the following result.
\begin{thm}
    There is no perfect plus state transfer in the complementary prism of a complete bipartite graph $K_{n,n}$ with respect to the signless Laplacian matrix for $n\geq 2.$ 
\end{thm}

 \section*{Disclosure statement}
  No potential conflict of interest was reported by the author(s).

\section*{Acknowledgements}
 We sincerely thank the reviewers for their insightful comments and valuable suggestions to improve the manuscript. 
S. Mohapatra is supported by the Department of Science and Technology, Government of India (INSPIRE Fellowship: IF210209).

%%%%~~Bibliography~~%%%%%%

\bibliographystyle{abbrv}
\bibliography{references}

\end{document}